\documentclass[11pt]{article}

\usepackage[letterpaper,margin=1.1in]{geometry}
\usepackage[T1]{fontenc}
\usepackage{amsmath,amssymb,amsthm}
\usepackage{newpxtext}
\usepackage{newpxmath}
\usepackage{mathtools}
\usepackage{microtype}
\usepackage[hidelinks]{hyperref}
\usepackage[nameinlink,noabbrev]{cleveref}
\usepackage{enumitem}

\numberwithin{equation}{section}
\allowdisplaybreaks

\newtheorem{theorem}{Theorem}[section]
\newtheorem{proposition}[theorem]{Proposition}

\newtheorem{corollary}[theorem]{Corollary}
\theoremstyle{definition}
\newtheorem{definition}[theorem]{Definition}
\newtheorem{example}[theorem]{Example}
\theoremstyle{remark}
\newtheorem{remark}[theorem]{Remark}

\crefname{theorem}{theorem}{theorems}
\crefname{proposition}{proposition}{propositions}
\crefname{lemma}{lemma}{lemmas}
\crefname{corollary}{corollary}{corollaries}
\crefname{definition}{definition}{definitions}
\crefname{example}{example}{examples}
\crefname{remark}{remark}{remarks}

\newcommand{\SP}{\mathrm{SP}}

\newcommand{\cQ}{\mathcal{Q}}
\newcommand{\Ht}{\mathsf{H}}
\newcommand{\Pf}{\operatorname{Pf}}

\newcommand{\ZZ}{\mathbb{Z}}
\newcommand{\EE}{\mathbb{E}}
\newcommand{\Var}{\operatorname{Var}}
\newcommand{\Prob}{\mathbb{P}}

\title{A Two-Color Lift of the Shifted \texorpdfstring{$t$}{t}-Schur Measure}
\author{S.-J. Lee}
\date{}

\begin{document}

\maketitle

\begin{abstract}
At the specialization $t=-q$, $q\geq0$, the shifted $t$-Schur
function associated with the modified odd Greaves--Jing--Zhu operator is
$Q_\lambda[X+qX]$.  Instead of merging the two alphabets $X$ and $qX$, we
insert an intermediate strict partition between the two corresponding
half-vertex operators.  This gives a two-color lift of the shifted Schur
measure on pairs $\mu\subseteq\lambda$ with weight
\[
        Q_\mu(qX)Q_{\lambda/\mu}(X)P_\lambda(Y).
\]
We compute the normalization and both marginals, identify an explicit
Markov transition kernel, prove a semigroup property, and show that the
two color volumes $|\mu|$ and $|\lambda|-|\mu|$ are independent.  We also
realize the model as a two-time shifted Schur process and write its
Pfaffian correlation kernel in Vuleti\'c's convention.  Rectangular
specializations give closed formulas and Gaussian limits for the color
volumes.
\end{abstract}

\tableofcontents

\section{Introduction}

The modified odd Greaves--Jing--Zhu operator was introduced in
\cite{Lee} as an odd-power-sum analogue of the Greaves--Jing--Zhu
construction \cite{GJZ}.  In that normalization, the shifted $t$-Schur
weight has one-time distribution
\begin{equation}\label{eq:intro-one-time}
        \Prob_{-q}^{X,Y}(\lambda)
        =\frac{Q_\lambda[X+qX]P_\lambda(Y)}{\Ht(X+qX;Y)},
        \qquad q\geq0.
\end{equation}
As a distribution of the final strict partition $\lambda$, this is simply
the shifted Schur measure with first specialization $X+qX$.  Hence the
one-time Pfaffian structure follows from the classical shifted Schur
measure \cite{TW,Matsumoto}.

The modified operator, however, contains more information.  Since only
odd power sums occur,
\[
        1-(-q)^n=1+q^n\qquad(n\text{ odd}),
\]
and the creation half of the modified operator factorizes as
\begin{equation}\label{eq:intro-factorization}
        \Gamma_-^{(-q)}(X)=\Gamma_-(X)\Gamma_-(qX).
\end{equation}
If the two factors in \eqref{eq:intro-factorization} are multiplied
without further structure, one obtains \eqref{eq:intro-one-time}.  If an
intermediate strict partition is retained, one obtains a two-color
object
\begin{equation}\label{eq:intro-pair}
        \mu\subseteq\lambda.
\end{equation}
The first color occupies the shape $\mu$, and the second color occupies
the skew shape $\lambda/\mu$.

The central object of this note is the probability measure
\begin{equation}\label{eq:intro-measure}
        \Prob_q^{X,Y}(\mu,
        \lambda)
        =\frac{Q_\mu(qX)Q_{\lambda/\mu}(X)P_\lambda(Y)}{\Ht(qX;Y)\Ht(X;Y)},
        \qquad \mu\subseteq\lambda.
\end{equation}
The $\lambda$-marginal is \eqref{eq:intro-one-time}, but the joint law
has additional structure.  It admits the Markov factorization
\begin{equation}\label{eq:intro-markov-factorization}
        \Prob_q^{X,Y}(\mu,\lambda)
        =\pi_{qX,Y}(\mu)K_X^Y(\mu,\lambda),
\end{equation}
where
\begin{equation}\label{eq:intro-K}
        K_X^Y(\mu,\lambda)
        =\frac{Q_{\lambda/\mu}(X)P_\lambda(Y)}{\Ht(X;Y)P_\mu(Y)}.
\end{equation}
Moreover,
\begin{equation}\label{eq:intro-semigroup}
        K_X^Y K_Z^Y=K_{X+Z}^Y.
\end{equation}
This semigroup relation is the probabilistic form of the half-vertex
operator identity $\Gamma_-(X)\Gamma_-(Z)=\Gamma_-(X+Z)$.

The color-volume statistics are especially simple.  Define
\begin{equation}\label{eq:intro-color-volumes}
        B=|\mu|,
        \qquad
        R=|\lambda|-|\mu|.
\end{equation}
Then
\begin{equation}\label{eq:intro-BR-pgf}
        \EE[u^Bv^R]
        =\frac{\Ht(uqX;Y)}{\Ht(qX;Y)}
        \frac{\Ht(vX;Y)}{\Ht(X;Y)}.
\end{equation}
Thus $B$ and $R$ are independent.  Their cumulants are
\begin{equation}\label{eq:intro-B-cumulants}
        \kappa_m(B)
        =2\sum_{\substack{n\geq1\\n\text{ odd}}}
        n^{m-1}q^n p_n(X)p_n(Y),
\end{equation}
\begin{equation}\label{eq:intro-R-cumulants}
        \kappa_m(R)
        =2\sum_{\substack{n\geq1\\n\text{ odd}}}
        n^{m-1}p_n(X)p_n(Y).
\end{equation}

Finally, the joint process is a two-time specialization of Vuleti\'c's
shifted Schur process \cite{Vuletic}.  Therefore the point process
\begin{equation}\label{eq:intro-two-time-point-process}
        \mathfrak X(\mu,\lambda)
        =\{(k,1):k\in\mu\}\cup\{(k,2):k\in\lambda\}
\end{equation}
is Pfaffian.  The time symbols are
\begin{equation}\label{eq:intro-J1J2}
        J_1(z)=F_Y(z)F_{qX}(z^{-1}),
        \qquad
        J_2(z)=F_Y(z)F_{qX}(z^{-1})F_X(z^{-1}).
\end{equation}
The precise kernel is stated in \cref{sec:pfaffian}.

\section{Preliminaries}

We use the standard notation for Schur $Q$- and $P$-functions; see
\cite[Chapter III, Section 8]{Macdonald} and \cite{Stembridge}.  Let
$\SP$ denote the set of strict partitions.  We use the notation
$\mu\subseteq\lambda$ to mean $\mu_i\leq\lambda_i$ for all $i$, after
appending zero parts.  For an alphabet or specialization $A$, set
\begin{equation}\label{eq:FA}
        F_A(z)
        =\prod_{a\in A}\frac{1+az}{1-az}
        =\exp\left(
        2\sum_{\substack{n\geq1\\n\text{ odd}}}
        \frac{p_n(A)}{n}z^n
        \right).
\end{equation}
The Schur $Q$- and $P$-functions are normalized by
\begin{equation}\label{eq:PQ-normalization}
        P_\lambda=2^{-\ell(\lambda)}Q_\lambda.
\end{equation}
The Schur $Q/P$ Cauchy kernel is
\begin{equation}\label{eq:HAB}
        \Ht(A;B)
        =\sum_{\lambda\in\SP}Q_\lambda(A)P_\lambda(B)
        =\exp\left(
        2\sum_{\substack{n\geq1\\n\text{ odd}}}
        \frac{p_n(A)p_n(B)}{n}\right).
\end{equation}
For finite alphabets,
\begin{equation}\label{eq:HAB-product}
        \Ht(A;B)=\prod_{a\in A,b\in B}\frac{1+ab}{1-ab}.
\end{equation}

We shall use the skew $Q$-functions determined by the branching identity
\begin{equation}\label{eq:branching}
        Q_\lambda(A+B)
        =\sum_{\mu\subseteq\lambda}Q_{\lambda/\mu}(A)Q_\mu(B).
\end{equation}
The corresponding skew Cauchy identity is
\begin{equation}\label{eq:skew-cauchy}
        \sum_{\lambda\supseteq\mu}Q_{\lambda/\mu}(A)P_\lambda(B)
        =\Ht(A;B)P_\mu(B).
\end{equation}

We use half-vertex-operator notation only through its matrix elements.
The creation operator $\Gamma_-(A)$ is characterized by
\begin{equation}\label{eq:Gamma-matrix-elements}
        \langle\lambda|\Gamma_-(A)|\mu\rangle=Q_{\lambda/\mu}(A),
        \qquad
        \langle\mu|\Gamma_-(A)|0\rangle=Q_\mu(A).
\end{equation}
This convention is compatible with the neutral-fermion realization of
Schur $Q$-functions \cite{Jing} and with the modified odd operator
construction of \cite{Lee,LeeMixed}.

\begin{remark}\label{rem:formal-vs-probability}
All identities below hold algebraically for specializations for which the
series are meaningful.  For probabilistic language we assume $q\geq0$ and
finite nonnegative alphabets $X=(x_i)$ and $Y=(y_j)$ satisfying
\begin{equation}\label{eq:convergence}
        \max\{1,q\}x_i y_j<1\qquad\text{for every }i,j.
\end{equation}
This ensures that all displayed normalizing products are finite and that
the weights are nonnegative.
\end{remark}

\section{Two-color factorization}

In the normalization of \cite{Lee,LeeTransition}, the shifted
$t$-Schur function is obtained from $Q_\lambda$ by the plethystic change
$X\mapsto X-tX$.  Therefore, at $t=-q$,
\begin{equation}\label{eq:Qt-negative-q}
        \cQ_\lambda(X;-q)=Q_\lambda[X+qX].
\end{equation}
Indeed, for each odd $n$,
\[
        p_n[X-(-q)X]=(1+q^n)p_n(X)=p_n[X+qX].
\]
On the half-vertex-operator level this is the factorization
\begin{equation}\label{eq:half-vertex-factorization}
        \Gamma_-^{(-q)}(X)=\Gamma_-(X)\Gamma_-(qX),
\end{equation}
which is the $t=-q$ case of the diagonal odd-power-sum scaling appearing
in the modified operator construction \cite{Lee,LeeMixed}.

\begin{proposition}[Intermediate strict partition]\label{prop:intermediate}
For every strict partition $\lambda$,
\begin{equation}\label{eq:intermediate-branching}
        Q_\lambda[X+qX]
        =\sum_{\mu\subseteq\lambda}Q_\mu(qX)Q_{\lambda/\mu}(X).
\end{equation}
\end{proposition}

\begin{proof}
This is the branching identity \eqref{eq:branching} with $A=X$ and
$B=qX$.  Equivalently, using \eqref{eq:Gamma-matrix-elements}, one inserts
the identity operator between the two commuting creation operators in
\eqref{eq:half-vertex-factorization}:
\[
\begin{aligned}
        \langle \lambda|\Gamma_-(X)\Gamma_-(qX)|0\rangle
        &=\sum_{\mu\in\SP}
        \langle \lambda|\Gamma_-(X)|\mu\rangle
        \langle \mu|\Gamma_-(qX)|0\rangle\\
        &=\sum_{\mu\subseteq\lambda}Q_{\lambda/\mu}(X)Q_\mu(qX).
\end{aligned}
\]
\end{proof}

\begin{definition}\label{def:colors}
In the pair $\mu\subseteq\lambda$, we call $\mu$ the blue shape and
$\lambda/\mu$ the red skew shape.  The words blue and red are external
color labels.  They are independent of the usual primed and unprimed
entries appearing in shifted semistandard tableaux.
\end{definition}

\section{The joint measure and its marginals}

\begin{definition}\label{def:joint-measure}
Assume the positivity and convergence conditions in \cref{rem:formal-vs-probability}.  Define
\begin{equation}\label{eq:joint-measure}
        \Prob_q^{X,Y}(\mu,\lambda)
        =\frac{Q_\mu(qX)Q_{\lambda/\mu}(X)P_\lambda(Y)}{Z_q(X,Y)},
        \qquad \mu\subseteq\lambda.
\end{equation}
Set the weight equal to zero when $\mu\nsubseteq\lambda$.
\end{definition}

\begin{theorem}[Normalization]\label{thm:normalization}
The normalizing constant is
\begin{equation}\label{eq:Zq}
        Z_q(X,Y)=\Ht(qX;Y)\Ht(X;Y)=\Ht(X+qX;Y).
\end{equation}
For finite alphabets this is
\begin{equation}\label{eq:Zq-product}
        Z_q(X,Y)
        =\prod_{i,j}
        \frac{1+x_i y_j}{1-x_i y_j}
        \frac{1+q x_i y_j}{1-q x_i y_j}.
\end{equation}
\end{theorem}

\begin{proof}
Using \cref{prop:intermediate},
\[
\begin{aligned}
        \sum_{\mu\subseteq\lambda}
        Q_\mu(qX)Q_{\lambda/\mu}(X)P_\lambda(Y)
        &=Q_\lambda[X+qX]P_\lambda(Y).
\end{aligned}
\]
Summing over $\lambda$ and applying the Cauchy identity gives
\[
        Z_q(X,Y)=\Ht(X+qX;Y).
\]
Since $p_n(X+qX)=p_n(X)+p_n(qX)$, the exponential form of $\Ht$ gives
\[
        \Ht(X+qX;Y)=\Ht(X;Y)\Ht(qX;Y).
\]
The product formula follows from \eqref{eq:HAB-product}.
\end{proof}

\begin{theorem}[Marginals]\label{thm:marginals}
The $\lambda$-marginal is
\begin{equation}\label{eq:lambda-marginal}
        \Prob_q^{X,Y}(\lambda)
        =\frac{Q_\lambda[X+qX]P_\lambda(Y)}{\Ht(X+qX;Y)}.
\end{equation}
The $\mu$-marginal is
\begin{equation}\label{eq:mu-marginal}
        \Prob_q^{X,Y}(\mu)
        =\frac{Q_\mu(qX)P_\mu(Y)}{\Ht(qX;Y)}.
\end{equation}
\end{theorem}

\begin{proof}
The first formula follows by summing \eqref{eq:joint-measure} over
$\mu$ and using \eqref{eq:intermediate-branching}.  For the second,
fix $\mu$ and use the skew Cauchy identity \eqref{eq:skew-cauchy}:
\[
\begin{aligned}
        \sum_{\lambda\supseteq\mu}
        Q_{\lambda/\mu}(X)P_\lambda(Y)
        =\Ht(X;Y)P_\mu(Y).
\end{aligned}
\]
Therefore
\[
\begin{aligned}
        \Prob_q^{X,Y}(\mu)
        &=\frac{Q_\mu(qX)}{\Ht(qX;Y)\Ht(X;Y)}
        \sum_{\lambda\supseteq\mu}Q_{\lambda/\mu}(X)P_\lambda(Y)\\
        &=\frac{Q_\mu(qX)P_\mu(Y)}{\Ht(qX;Y)}.
\end{aligned}
\]
\end{proof}

\begin{remark}\label{rem:not-new-one-time}
The one-time marginal \eqref{eq:lambda-marginal} is the shifted Schur
measure with first alphabet $X+qX$.  The additional object in this note
is the joint lift $(\mu,\lambda)$ and the color information it retains.
\end{remark}

\section{A Markov kernel}

\begin{definition}\label{def:markov-kernel}
For strict partitions $\mu,\lambda$ with $\mu\subseteq\lambda$, set
\begin{equation}\label{eq:K-definition}
        K_X^Y(\mu,\lambda)
        =\frac{Q_{\lambda/\mu}(X)P_\lambda(Y)}{\Ht(X;Y)P_\mu(Y)}.
\end{equation}
If $\mu\nsubseteq\lambda$, set $K_X^Y(\mu,\lambda)=0$.
\end{definition}

\begin{remark}\label{rem:zero-denominator}
If $P_\mu(Y)=0$, the formula is interpreted on the support of the
marginal distribution.  For formal identities, one may work over the
field of fractions.  For probability statements, rows with
$P_\mu(Y)=0$ have zero initial mass under \eqref{eq:mu-marginal}, and
may be assigned arbitrarily without changing the joint law.
\end{remark}

\begin{theorem}[Markov property]\label{thm:markov}
For every $\mu$ in the support of the marginal distribution,
\begin{equation}\label{eq:row-sum-one}
        \sum_{\lambda\in\SP}K_X^Y(\mu,\lambda)=1.
\end{equation}
Moreover,
\begin{equation}\label{eq:joint-markov-factorization}
        \Prob_q^{X,Y}(\mu,\lambda)
        =\pi_{qX,Y}(\mu)K_X^Y(\mu,\lambda),
\end{equation}
where
\begin{equation}\label{eq:pi-definition}
        \pi_{qX,Y}(\mu)=\frac{Q_\mu(qX)P_\mu(Y)}{\Ht(qX;Y)}.
\end{equation}
\end{theorem}

\begin{proof}
The row-sum identity is exactly the skew Cauchy identity divided by
$\Ht(X;Y)P_\mu(Y)$:
\[
        \sum_{\lambda\supseteq\mu}K_X^Y(\mu,\lambda)
        =\frac{1}{\Ht(X;Y)P_\mu(Y)}
        \sum_{\lambda\supseteq\mu}Q_{\lambda/\mu}(X)P_\lambda(Y)=1.
\]
The factorization follows by multiplying \eqref{eq:pi-definition} and
\eqref{eq:K-definition}:
\[
        \pi_{qX,Y}(\mu)K_X^Y(\mu,\lambda)
        =\frac{Q_\mu(qX)Q_{\lambda/\mu}(X)P_\lambda(Y)}{\Ht(qX;Y)\Ht(X;Y)}.
\]
By \cref{thm:normalization}, this is the joint probability
\eqref{eq:joint-measure}.
\end{proof}

\begin{theorem}[Semigroup property]\label{thm:semigroup}
For specializations $X$ and $Z$,
\begin{equation}\label{eq:semigroup}
        K_X^Y K_Z^Y=K_{X+Z}^Y.
\end{equation}
Equivalently,
\begin{equation}\label{eq:semigroup-expanded}
        \sum_{\nu\in\SP}K_X^Y(\mu,\nu)K_Z^Y(\nu,\lambda)
        =K_{X+Z}^Y(\mu,\lambda).
\end{equation}
\end{theorem}

\begin{proof}
Using the definition of $K_X^Y$,
\[
\begin{aligned}
        &\sum_\nu K_X^Y(\mu,\nu)K_Z^Y(\nu,\lambda)\\
        &\quad =
        \frac{P_\lambda(Y)}{\Ht(X;Y)\Ht(Z;Y)P_\mu(Y)}
        \sum_\nu Q_{\nu/\mu}(X)Q_{\lambda/\nu}(Z).
\end{aligned}
\]
The branching identity for skew $Q$-functions gives
\begin{equation}\label{eq:skew-branching}
        \sum_\nu Q_{\nu/\mu}(X)Q_{\lambda/\nu}(Z)
        =Q_{\lambda/\mu}(X+Z).
\end{equation}
Also, $\Ht(X+Z;Y)=\Ht(X;Y)\Ht(Z;Y)$.  Substitution yields
\[
        \sum_\nu K_X^Y(\mu,\nu)K_Z^Y(\nu,\lambda)
        =\frac{Q_{\lambda/\mu}(X+Z)P_\lambda(Y)}{\Ht(X+Z;Y)P_\mu(Y)}.
\]
The right-hand side is $K_{X+Z}^Y(\mu,\lambda)$.
\end{proof}

\begin{corollary}\label{cor:marginal-evolution}
The shifted Schur measure evolves by
\begin{equation}\label{eq:pi-evolution}
        \pi_{qX,Y}K_X^Y=\pi_{qX+X,Y}.
\end{equation}
\end{corollary}

\section{Color-volume statistics}

Define
\begin{equation}\label{eq:B-R-definition}
        B=|\mu|,
        \qquad
        R=|\lambda|-|\mu|.
\end{equation}
Thus $B$ is the blue volume and $R$ is the red volume.

\begin{theorem}[Joint probability generating function]\label{thm:BR-pgf}
The joint probability generating function of $(B,R)$ is
\begin{equation}\label{eq:BR-pgf}
        \EE[u^Bv^R]
        =\frac{\Ht(uqX;Y)}{\Ht(qX;Y)}
        \frac{\Ht(vX;Y)}{\Ht(X;Y)}.
\end{equation}
Consequently, $B$ and $R$ are independent.
\end{theorem}

\begin{proof}
By homogeneity,
\[
        u^{|\mu|}Q_\mu(qX)=Q_\mu(uqX),
\]
and
\[
        v^{|\lambda|-|\mu|}Q_{\lambda/\mu}(X)=Q_{\lambda/\mu}(vX).
\]
Therefore the unnormalized generating function is
\[
\begin{aligned}
        &\sum_{\mu\subseteq\lambda}
        u^{|\mu|}v^{|\lambda|-|\mu|}
        Q_\mu(qX)Q_{\lambda/\mu}(X)P_\lambda(Y)\\
        &\quad=
        \sum_{\mu\subseteq\lambda}
        Q_\mu(uqX)Q_{\lambda/\mu}(vX)P_\lambda(Y).
\end{aligned}
\]
Summing first over $\mu$ gives $Q_\lambda[uqX+vX]$, and then Cauchy's
identity gives $\Ht(uqX+vX;Y)$.  Since
\[
        \Ht(uqX+vX;Y)=\Ht(uqX;Y)\Ht(vX;Y),
\]
and $Z_q=\Ht(qX;Y)\Ht(X;Y)$, we obtain \eqref{eq:BR-pgf}.  The right-hand
side is the product of a function of $u$ and a function of $v$, so $B$ and
$R$ are independent.
\end{proof}

\begin{theorem}[Cumulants]\label{thm:color-cumulants}
The $m$th cumulants of $B$ and $R$ are
\begin{equation}\label{eq:B-cumulants}
        \kappa_m(B)
        =2\sum_{\substack{n\geq1\\n\text{ odd}}}
        n^{m-1}q^n p_n(X)p_n(Y),
\end{equation}
\begin{equation}\label{eq:R-cumulants}
        \kappa_m(R)
        =2\sum_{\substack{n\geq1\\n\text{ odd}}}
        n^{m-1}p_n(X)p_n(Y).
\end{equation}
Consequently,
\begin{equation}\label{eq:total-cumulants}
        \kappa_m(|\lambda|)
        =2\sum_{\substack{n\geq1\\n\text{ odd}}}
        n^{m-1}(1+q^n)p_n(X)p_n(Y).
\end{equation}
\end{theorem}

\begin{proof}
Set $u=e^s$ in the first factor of \eqref{eq:BR-pgf}.  Then
\[
        \log\EE[e^{sB}]
        =\log\Ht(e^s qX;Y)-\log\Ht(qX;Y).
\]
Using the exponential formula for $\Ht$,
\[
        \log\EE[e^{sB}]
        =2\sum_{\substack{n\geq1\\n\text{ odd}}}
        \frac{(e^{ns}-1)q^n p_n(X)p_n(Y)}{n}.
\]
Taking the $m$th derivative at $s=0$ gives \eqref{eq:B-cumulants}.  The
proof for $R$ is identical with $q$ removed.  Since $|\lambda|=B+R$ and
$B,R$ are independent, cumulants add, giving \eqref{eq:total-cumulants}.
\end{proof}

\subsection{Finite alphabet decomposition}

For $0\leq a<1$, define a random variable $\xi_a$ on $\ZZ_{\geq0}$ by
\begin{equation}\label{eq:xi-distribution-zero}
        \Prob(\xi_a=0)=\frac{1-a}{1+a},
\end{equation}
\begin{equation}\label{eq:xi-distribution-positive}
        \Prob(\xi_a=k)=2\frac{1-a}{1+a}a^k,
        \qquad k\geq1.
\end{equation}
When $a=0$, this definition gives the degenerate variable $\xi_0=0$.
Its probability generating function is
\begin{equation}\label{eq:xi-pgf}
        \EE[r^{\xi_a}]
        =\frac{1-a}{1+a}\frac{1+ar}{1-ar}.
\end{equation}

\begin{corollary}\label{cor:finite-alphabet-decomposition}
Let $X=(x_1,\ldots,x_M)$ and $Y=(y_1,\ldots,y_N)$.  Then
\begin{equation}\label{eq:B-independent-sum}
        B\overset{d}{=}
        \sum_{i=1}^M\sum_{j=1}^N \xi_{q x_i y_j},
\end{equation}
and
\begin{equation}\label{eq:R-independent-sum}
        R\overset{d}{=}
        \sum_{i=1}^M\sum_{j=1}^N \widetilde\xi_{x_i y_j},
\end{equation}
where all random variables on the right-hand side are independent.
\end{corollary}

\begin{proof}
For finite alphabets,
\[
        \frac{\Ht(uqX;Y)}{\Ht(qX;Y)}
        =\prod_{i,j}
        \frac{1-qx_i y_j}{1+qx_i y_j}
        \frac{1+u qx_i y_j}{1-u qx_i y_j}.
\]
Each factor is the probability generating function \eqref{eq:xi-pgf}
with $a=qx_i y_j$.  This proves the formula for $B$.  The proof for $R$
is the same with $a=x_i y_j$.
\end{proof}

\section{Conditional color distribution}

The conditional law of the blue shape given the final strict partition is
explicit.

\begin{proposition}\label{prop:conditional-shape}
For fixed $\lambda$ in the support of the $\lambda$-marginal,
\begin{equation}\label{eq:conditional-shape}
        \Prob_q(\mu\mid\lambda)
        =\frac{Q_\mu(qX)Q_{\lambda/\mu}(X)}{Q_\lambda[X+qX]},
        \qquad \mu\subseteq\lambda.
\end{equation}
Consequently,
\begin{equation}\label{eq:conditional-B-pgf}
        \EE[u^B\mid\lambda]
        =\frac{Q_\lambda[X+uqX]}{Q_\lambda[X+qX]}.
\end{equation}
In particular,
\begin{equation}\label{eq:conditional-B-mean}
        \EE[B\mid\lambda]
        =q\frac{\partial}{\partial q}\log Q_\lambda[X+qX],
\end{equation}
and
\begin{equation}\label{eq:conditional-B-variance}
        \Var(B\mid\lambda)
        =\left(q\frac{\partial}{\partial q}\right)^2
        \log Q_\lambda[X+qX].
\end{equation}
\end{proposition}

\begin{proof}
The conditional probability follows by dividing the joint weight by its
$\lambda$-marginal.  Then
\[
\begin{aligned}
        \EE[u^B\mid\lambda]
        &=\frac{\sum_{\mu\subseteq\lambda}
        u^{|\mu|}Q_\mu(qX)Q_{\lambda/\mu}(X)}{Q_\lambda[X+qX]}\\
        &=\frac{\sum_{\mu\subseteq\lambda}
        Q_\mu(uqX)Q_{\lambda/\mu}(X)}{Q_\lambda[X+qX]}\\
        &=\frac{Q_\lambda[X+uqX]}{Q_\lambda[X+qX]}.
\end{aligned}
\]
The conditional mean is obtained by applying $u\partial_u$ to
$\EE[u^B\mid\lambda]$ at $u=1$.  The conditional variance is obtained by
applying $(u\partial_u)^2$ to
$\log\EE[u^B\mid\lambda]$ at $u=1$.  Since the parameter $u$ appears only
through $uq$, these operations are equivalent to applying $q\partial_q$
and $(q\partial_q)^2$ to the displayed logarithm.
\end{proof}

\begin{example}\label{ex:lambda-one}
For $\lambda=(1)$, one has
\[
        Q_{(1)}[X+qX]=(1+q)Q_{(1)}(X).
\]
Thus
\[
        \Prob(B=0\mid\lambda=(1))=\frac{1}{1+q},
        \qquad
        \Prob(B=1\mid\lambda=(1))=\frac{q}{1+q}.
\]
This is the smallest example showing that $q$ is the blue color weight.
\end{example}

\section{A tableau interpretation}

The weight in \eqref{eq:joint-measure} has a direct combinatorial
interpretation in terms of shifted semistandard tableaux
\cite{Stembridge,Macdonald}.  The factor $Q_\mu(qX)$ is the generating
function for shifted semistandard tableaux of shape $\mu$ with alphabet
$qX$.  The factor $Q_{\lambda/\mu}(X)$ is the generating function for
shifted semistandard tableaux of skew shape $\lambda/\mu$ with alphabet
$X$.  The factor $P_\lambda(Y)=2^{-\ell(\lambda)}Q_\lambda(Y)$ supplies
the corresponding $P$-normalization on the final shape.

Thus the unnormalized measure enumerates triples
\begin{equation}\label{eq:tableau-triples}
        (T^{\mathrm{blue}},T^{\mathrm{red}},U),
\end{equation}
where $T^{\mathrm{blue}}$ has shape $\mu$, $T^{\mathrm{red}}$ has skew
shape $\lambda/\mu$, and $U$ has shape $\lambda$ and is weighted with the
$P$-normalization.  The blue entries carry weights from $qX$, while the
red entries carry weights from $X$.

\begin{remark}\label{rem:external-color}
The blue/red color introduced here is an external color.  It should not
be confused with the primed/unprimed markings appearing in the usual
combinatorics of shifted tableaux.
\end{remark}

\section{Pfaffian correlation functions}\label{sec:pfaffian}

The joint law is a two-time shifted Schur process.  We use Vuleti\'c's
notation for the shifted Schur process \cite{Vuletic}, the shifted analogue
of the Schur-process formalism of Okounkov and Reshetikhin \cite{OR}, in
this section.

\subsection{Embedding into the shifted Schur process}

Vuleti\'c's shifted Schur process is determined by a sequence of
specializations
\begin{equation}\label{eq:vuletic-specializations-general}
        \rho=(\rho_0^+,\rho_1^-,\rho_1^+,\ldots,\rho_T^-).
\end{equation}
We take $T=2$ and specialize
\begin{equation}\label{eq:our-specializations}
        \rho_0^+=qX,
        \qquad
        \rho_1^-=0,
        \qquad
        \rho_1^+=X,
        \qquad
        \rho_2^-=Y.
\end{equation}
Then the two visible strict partitions in the process are
\begin{equation}\label{eq:two-times}
        \lambda^{(1)}=\mu,
        \qquad
        \lambda^{(2)}=\lambda.
\end{equation}
The intermediate partition associated with $\rho_1^-=0$ is forced to be
$\lambda^{(1)}$, and the process weight becomes
\begin{equation}\label{eq:process-weight-specialized}
        Q_{\lambda^{(1)}}(qX)Q_{\lambda^{(2)}/\lambda^{(1)}}(X)
        P_{\lambda^{(2)}}(Y),
\end{equation}
which is exactly the numerator in \eqref{eq:joint-measure}.

\subsection{The two-time point process}

Define
\begin{equation}\label{eq:point-process}
        \mathfrak X(\mu,\lambda)
        =\{(k,1):k\in\mu\}\cup\{(k,2):k\in\lambda\}
        \subset \ZZ_{>0}\times\{1,2\}.
\end{equation}
For a finite subset
\[
        S=\{(x_1,t_1),\ldots,(x_n,t_n)\}
        \subset\ZZ_{>0}\times\{1,2\},
\]
define
\begin{equation}\label{eq:process-correlation}
        \rho(S)=\Prob_q^{X,Y}\{S\subseteq\mathfrak X(\mu,\lambda)\}.
\end{equation}

\subsection{The extended kernel}

For $a,b\in\{1,2\}$ and $u,v\in\ZZ$, define $K_{u,v}(a,b)$ as the
coefficient of $z^u w^v$ in
\begin{equation}\label{eq:extended-kernel-series}
        \frac{z-w}{2(z+w)}J_a(z)J_b(w),
\end{equation}
expanded in the region
\begin{equation}\label{eq:expansion-regions}
        |z|>|w|\quad\text{if }a\geq b,
        \qquad
        |z|<|w|\quad\text{if }a<b.
\end{equation}
Here Vuleti\'c's symbols are
\begin{equation}\label{eq:J1}
        J_1(z)=F_Y(z)F_{qX}(z^{-1}),
\end{equation}
\begin{equation}\label{eq:J2}
        J_2(z)=F_Y(z)F_{qX}(z^{-1})F_X(z^{-1}).
\end{equation}
Indeed, Vuleti\'c's general formula is
\begin{equation}\label{eq:Vuletic-J-general}
        J(t,z)=
        \prod_{t\leq m}F(\rho_m^-;z)
        \prod_{m\leq t-1}F(\rho_m^+;z^{-1}),
\end{equation}
and \eqref{eq:J1}--\eqref{eq:J2} follow by substituting
\eqref{eq:our-specializations}.

\begin{theorem}[Two-time Pfaffian kernel]\label{thm:two-time-pfaffian}
Let
\[
        S=\{(x_1,t_1),\ldots,(x_n,t_n)\}
        \subset\ZZ_{>0}\times\{1,2\}.
\]
Set $x_i'=x_{2n-i+1}$ and $t_i'=t_{2n-i+1}$ when an index is reflected
across the middle.  Define a $2n\times2n$ skew-symmetric matrix $M_S$ by
specifying the entries above the diagonal:
\begin{equation}\label{eq:MS-definition}
        (M_S)_{ij}=
        \begin{cases}
        K_{x_i,x_j}(t_i,t_j), & 1\leq i<j\leq n,\\
        (-1)^{x_{2n-j+1}}K_{x_i,-x_{2n-j+1}}(t_i,t_{2n-j+1}),
        & 1\leq i\leq n<j\leq2n,\\
        (-1)^{x_{2n-i+1}+x_{2n-j+1}}
        K_{-x_{2n-i+1},-x_{2n-j+1}}(t_{2n-i+1},t_{2n-j+1}),
        & n<i<j\leq2n.
        \end{cases}
\end{equation}
Then
\begin{equation}\label{eq:pfaffian-process-correlation}
        \rho(S)=\Pf(M_S).
\end{equation}
\end{theorem}

\begin{proof}
This is Vuleti\'c's Pfaffian theorem for the shifted Schur process
\cite{Vuletic} applied
to the specializations \eqref{eq:our-specializations}.  The specialization
$\rho_1^-=0$ forces the hidden intermediate partition in Vuleti\'c's
process to coincide with the first visible partition.  Hence the process
weight is exactly \eqref{eq:process-weight-specialized}, and the
normalization is $\Ht(qX;Y)\Ht(X;Y)$ by \cref{thm:normalization}.  The
coefficient rule, the signs in \eqref{eq:MS-definition}, and the
expansion regions are precisely those in Vuleti\'c's theorem.  Substituting
\eqref{eq:our-specializations} into \eqref{eq:Vuletic-J-general} gives
\eqref{eq:J1} and \eqref{eq:J2}.
\end{proof}

\begin{remark}\label{rem:one-time-conventions}
The one-time kernel of the shifted Schur measure is often written in
Matsumoto's convention with symbol $F_A(z)F_B(-z^{-1})$.  The process
formula above uses Vuleti\'c's convention.  The two descriptions are
equivalent after the standard change of variables and sign convention,
but the formulas should not be mixed without translating conventions.
\end{remark}

\subsection{Largest parts}

For $h_1,h_2\geq0$, let
\begin{equation}\label{eq:E-h1-h2}
        E_{h_1,h_2}
        =\{(k,1):k\geq h_1+1\}
        \cup
        \{(k,2):k\geq h_2+1\}.
\end{equation}
Here the first coordinate is the part size and the second coordinate is
the time.  The event $\{\mu_1\leq h_1,\lambda_1\leq h_2\}$ is the gap
event $\mathfrak X(\mu,\lambda)\cap E_{h_1,h_2}=\varnothing$.

The scalar kernel $K_{u,v}(a,b)$ may be packaged into the usual
$2\times2$ matrix kernel
\begin{equation}\label{eq:matrix-kernel-for-fredholm}
        \mathbf K((x,t),(y,s))=
        \begin{pmatrix}
        K_{x,y}(t,s) & (-1)^yK_{x,-y}(t,s)\\
        (-1)^xK_{-x,y}(t,s) & (-1)^{x+y}K_{-x,-y}(t,s)
        \end{pmatrix}.
\end{equation}
Let
\begin{equation}\label{eq:skew-identity-kernel}
        J((x,t),(y,s))=
        \mathbf 1_{(x,t)=(y,s)}
        \begin{pmatrix}0&1\\-1&0\end{pmatrix}
\end{equation}
be the standard skew identity kernel.  Then the standard
Fredholm-Pfaffian gap formula for a Pfaffian point process gives
\cite{BorodinRains}
\begin{equation}\label{eq:largest-parts-fredholm}
        \Prob_q^{X,Y}(\mu_1\leq h_1,\lambda_1\leq h_2)
        =\Pf\bigl(J-\mathbf K\bigr)_{\ell^2(E_{h_1,h_2})}.
\end{equation}
Equivalently, \eqref{eq:largest-parts-fredholm} is the inclusion-exclusion
expansion obtained from the correlations in \cref{thm:two-time-pfaffian}.

\section{Rectangular specialization}

Let
\begin{equation}\label{eq:rectangular}
        X=(\underbrace{x,\ldots,x}_{M}),
        \qquad
        Y=(\underbrace{y,\ldots,y}_{N}),
        \qquad
        L=MN.
\end{equation}
Assume $\max\{1,q\}xy<1$.

\begin{proposition}\label{prop:rectangular-Z}
In the rectangular specialization,
\begin{equation}\label{eq:rectangular-Z}
        Z_q
        =\left(\frac{1+xy}{1-xy}\right)^L
        \left(\frac{1+qxy}{1-qxy}\right)^L.
\end{equation}
The two time symbols are
\begin{equation}\label{eq:rectangular-J1}
        J_1(z)
        =\left(\frac{1+yz}{1-yz}\right)^N
        \left(\frac{z+qx}{z-qx}\right)^M,
\end{equation}
\begin{equation}\label{eq:rectangular-J2}
        J_2(z)
        =J_1(z)
        \left(\frac{z+x}{z-x}\right)^M.
\end{equation}
\end{proposition}

\begin{proof}
The formula for $Z_q$ is \eqref{eq:Zq-product} with all $x_i=x$ and all
$y_j=y$.  The formulas for $J_1$ and $J_2$ follow from
\eqref{eq:J1}--\eqref{eq:J2} and
\[
        F_Y(z)=\left(\frac{1+yz}{1-yz}\right)^N,
        \qquad
        F_{qX}(z^{-1})=\left(\frac{z+qx}{z-qx}\right)^M,
\]
\[
        F_X(z^{-1})=\left(\frac{z+x}{z-x}\right)^M.
\]
\end{proof}

In this specialization, the extended kernel may be written as a double
contour integral:
\begin{equation}\label{eq:rectangular-kernel-integral}
        K_{u,v}(a,b)
        =\frac{1}{(2\pi i)^2}\oint\oint
        \frac{z-w}{2(z+w)}J_a(z)J_b(w)
        \frac{dz\,dw}{z^{u+1}w^{v+1}}.
\end{equation}
The contours are nested so that $|z|>|w|$ if $a\geq b$ and $|z|<|w|$ if
$a<b$.  They are chosen in an annulus in which the Laurent expansions
above are valid, separating the poles at $qx$ and $x$ from the pole at
$1/y$ in the usual way.

\subsection{Closed formulas for the color volumes}

Set
\begin{equation}\label{eq:a-b-parameters}
        a=qxy,
        \qquad
        b=xy.
\end{equation}
Then
\begin{equation}\label{eq:B-sum-rectangular}
        B\overset{d}{=}\sum_{r=1}^L\xi_a,
        \qquad
        R\overset{d}{=}\sum_{r=1}^L\xi_b,
\end{equation}
with the two sums independent.

\begin{proposition}\label{prop:B-distribution}
For $k\geq0$,
\begin{equation}\label{eq:B-distribution}
        \Prob(B=k)
        =\left(\frac{1-a}{1+a}\right)^L
        a^k
        \sum_{j=0}^{\min(L,k)}
        \binom{L}{j}\binom{L+k-j-1}{k-j}.
\end{equation}
The mean and variance are
\begin{equation}\label{eq:B-mean}
        \EE B=\frac{2La}{1-a^2},
\end{equation}
\begin{equation}\label{eq:B-variance}
        \Var(B)=\frac{2La(1+a^2)}{(1-a^2)^2}.
\end{equation}
The same formulas hold for $R$ after replacing $a$ by $b$.
\end{proposition}

\begin{proof}
The probability generating function of $B$ is
\[
        \EE[r^B]
        =\left(
        \frac{1-a}{1+a}\frac{1+ar}{1-ar}
        \right)^L.
\]
Expanding
\[
        (1+ar)^L=\sum_{j=0}^L\binom{L}{j}a^j r^j
\]
and
\[
        (1-ar)^{-L}=\sum_{m\geq0}\binom{L+m-1}{m}a^m r^m
\]
gives \eqref{eq:B-distribution}.  The mean and variance follow from
applying $r\partial_r$ and $(r\partial_r)^2$ to the logarithm of the
probability generating function at $r=1$.
\end{proof}

\begin{corollary}[Gaussian limit]\label{cor:gaussian-limit}
Assume $0<a,b<1$ are fixed and $L\to\infty$.  Then
\begin{equation}\label{eq:B-CLT}
        \frac{B-\dfrac{2La}{1-a^2}}
        {\sqrt{\dfrac{2La(1+a^2)}{(1-a^2)^2}}}
        \Longrightarrow N(0,1),
\end{equation}
\begin{equation}\label{eq:R-CLT}
        \frac{R-\dfrac{2Lb}{1-b^2}}
        {\sqrt{\dfrac{2Lb(1+b^2)}{(1-b^2)^2}}}
        \Longrightarrow N(0,1),
\end{equation}
and the two limiting normal random variables are independent.
\end{corollary}

\begin{proof}
By \eqref{eq:B-sum-rectangular}, $B$ and $R$ are sums of independent
identically distributed random variables with finite variance.  The
ordinary central limit theorem applies to each sum, and independence is
preserved in the limit.
\end{proof}


\begin{thebibliography}{99}

\bibitem{GJZ}
G. Greaves, N. Jing, and H. Zhu,
\newblock \textit{Vertex operators, infinite wedge representations, and correlation functions of the $t$-Schur measure},
\newblock arXiv:2602.14190, 2026.

\bibitem{Jing}
N. Jing,
\newblock \textit{Vertex operators, symmetric functions, and the spin group $\tilde S_n$},
\newblock J. Algebra \textbf{138} (1991), no. 2, 340--398.

\bibitem{Lee}
S.-J. Lee,
\newblock \textit{A Modified Greaves--Jing--Zhu Operator and a Shifted $t$-Gessel Formula},
\newblock arXiv:2606.22058, 2026.

\bibitem{LeeMixed}
S.-J. Lee,
\newblock \textit{Mixed Products of Modified Greaves--Jing--Zhu Operators},
\newblock arXiv:2606.28108, 2026.

\bibitem{LeeTransition}
S.-J. Lee,
\newblock \textit{Transition Matrices between Shifted $t$-Schur Bases and Cyclotomic Schur $Q$-Positivity},
\newblock arXiv:2606.28723, 2026.

\bibitem{Macdonald}
I. G. Macdonald,
\newblock \textit{Symmetric Functions and Hall Polynomials}, 2nd ed.,
\newblock Oxford University Press, 1995.

\bibitem{Matsumoto}
S. Matsumoto,
\newblock \textit{Correlation functions of the shifted Schur measure},
\newblock J. Math. Soc. Japan \textbf{57} (2005), no. 3, 619--637;
\newblock arXiv:math/0312373.

\bibitem{OR}
A. Okounkov and N. Reshetikhin,
\newblock \textit{Correlation function of Schur process with application to local geometry of a random 3-dimensional Young diagram},
\newblock J. Amer. Math. Soc. \textbf{16} (2003), no. 3, 581--603;
\newblock arXiv:math/0107056.

\bibitem{Stembridge}
J. R. Stembridge,
\newblock \textit{Shifted tableaux and the projective representations of symmetric groups},
\newblock Adv. Math. \textbf{74} (1989), no. 1, 87--134.

\bibitem{TW}
C. A. Tracy and H. Widom,
\newblock \textit{A limit theorem for shifted Schur measures},
\newblock Duke Math. J. \textbf{123} (2004), 171--208.

\bibitem{Vuletic}
M. Vuleti\'c,
\newblock \textit{Shifted Schur process and asymptotics of large random strict plane partitions},
\newblock Int. Math. Res. Not. IMRN \textbf{2007}, article ID rnm043;
\newblock arXiv:math-ph/0702068.

\bibitem{Worley}
D. R. Worley,
\newblock \textit{A Theory of Shifted Young Tableaux},
\newblock Ph.D. thesis, Massachusetts Institute of Technology, 1984.

\bibitem{Sagan}
B. E. Sagan,
\newblock \textit{Shifted tableaux, Schur $Q$-functions, and a conjecture of R. Stanley},
\newblock J. Combin. Theory Ser. A \textbf{45} (1987), no. 1, 62--103.

\bibitem{BorodinRains}
A. Borodin and E. M. Rains,
\newblock \textit{Eynard--Mehta theorem, Schur process, and their Pfaffian analogs},
\newblock J. Stat. Phys. \textbf{121} (2005), no. 3--4, 291--317;
\newblock arXiv:math-ph/0409059.

\bibitem{Schur}
I. Schur,
\newblock \textit{\"Uber die Darstellung der symmetrischen und der alternierenden Gruppe durch gebrochene lineare Substitutionen},
\newblock J. Reine Angew. Math. \textbf{139} (1911), 155--250.

\bibitem{Pragacz}
P. Pragacz,
\newblock \textit{Algebro-geometric applications of Schur $S$- and $Q$-polynomials},
\newblock in: Topics in Invariant Theory, Lecture Notes in Math. \textbf{1478}, Springer, 1991, 130--191.

\bibitem{StembridgeEnriched}
J. R. Stembridge,
\newblock \textit{Enriched $P$-partitions},
\newblock Trans. Amer. Math. Soc. \textbf{349} (1997), no. 2, 763--788.

\bibitem{Haiman}
M. Haiman,
\newblock \textit{On mixed insertion, symmetry, and shifted Young tableaux},
\newblock J. Combin. Theory Ser. A \textbf{50} (1989), no. 2, 196--225.

\bibitem{Rains}
E. M. Rains,
\newblock \textit{Correlation functions for symmetrized increasing subsequences},
\newblock arXiv:math/0006097, 2000.

\bibitem{BaikRains}
J. Baik and E. M. Rains,
\newblock \textit{Algebraic aspects of increasing subsequences},
\newblock Duke Math. J. \textbf{109} (2001), no. 1, 1--65;
\newblock arXiv:math/9905083.

\bibitem{Soshnikov}
A. Soshnikov,
\newblock \textit{Determinantal random point fields},
\newblock Russian Math. Surveys \textbf{55} (2000), no. 5, 923--975;
\newblock arXiv:math/0002099.

\bibitem{Borodin}
A. Borodin,
\newblock \textit{Determinantal point processes},
\newblock in: The Oxford Handbook of Random Matrix Theory, Oxford University Press, 2011, 231--249;
\newblock arXiv:0911.1153.

\bibitem{Okounkov}
A. Okounkov,
\newblock \textit{Infinite wedge and random partitions},
\newblock Selecta Math. (N.S.) \textbf{7} (2001), no. 1, 57--81;
\newblock arXiv:math/9907127.

\bibitem{OkounkovOlshanski}
A. Okounkov and G. Olshanski,
\newblock \textit{Shifted Schur functions},
\newblock St. Petersburg Math. J. \textbf{9} (1998), no. 2, 239--300;
\newblock arXiv:q-alg/9605042.

\end{thebibliography}
\end{document}